\documentclass{article}

\usepackage{amsxtra}
\usepackage{amsfonts}
\usepackage[latin1]{inputenc}
\usepackage{graphicx}
\usepackage{amsmath,amssymb,latexsym,amsthm,mathrsfs}
\usepackage[all]{xy}

\newtheorem{theorem}{Theorem}[section]

\newtheorem{proposition}[theorem]{Proposition}

\newtheorem{conjecture}[theorem]{Conjecture}

  \newcommand{\bC}{\mathbb{C}}   \newcommand{\bF}{\mathbb{F}}          
         \newcommand{\bZ}{\mathbb{Z}}

\newcommand{\sma}{\wedge} 

\newcommand{\TR}{\textnormal{TR}}
\newcommand{\TC}{\textnormal{TC}}

\title{On the algebraic $K$-theory of the coordinate axes over the integers}
\author{Vigleik Angeltveit\thanks{Vigleik Angeltveit was partially supported by an NSF All-Institutes Postdoctoral Fellowship administered by the Mathematical Sciences Research Institute through its core grant DMS-0441170} \, and Teena Gerhardt\thanks{Teena Gerhardt was partially supported by NSF Grant DMS-1007083}}

\begin{document}

\maketitle

\begin{abstract}
We show that the relative algebraic $K$-theory group $K_{2i}(\bZ[x,y]/(xy), (x,y))$ is free abelian of rank $1$ and that $K_{2i+1}(\bZ[x,y]/(xy), (x,y))$ is finite of order $(i!)^2$. We also find the group structure of $K_{2i+1}(\bZ[x,y]/(xy), (x,y))$ in low degrees.
\end{abstract}

\section{Introduction}
In general algebraic $K$-theory groups are difficult to compute, and relatively few computations exist in the literature. More than 30 years ago the relative algebraic $K$-theory group $K_q(\mathbb{Z}[x,y]/(xy), (x,y))$ was considered by Dennis and Krusemeyer \cite{DeKr79}, who computed this group when $q=2$.  About 20 years ago, work of Geller, Reid, and Weibel \cite{GeReWe89} showed that for every nonnegative integer $q$ the abelian group $K_q(\mathbb{Z}[x,y]/(xy), (x,y))$ has rank $0$ if $q$ is odd and $1$ if $q$ is even. We prove the following more precise result:

\begin{theorem} \label{thm:main}
For any $i \geq 0$
\begin{enumerate}
\item The abelian group $K_{2i}(\mathbb{Z}[x,y]/(xy), (x,y))$ is free of rank $1$.
\item The abelian group $K_{2i+1}(\mathbb{Z}[x,y]/(xy), (x,y))$ is finite of order $(i!)^2$.
\end{enumerate}
\end{theorem}
In \cite{He07}, Hesselholt evalulated the algebraic $K$-groups $K_*(k[x,y]/(xy), (x,y))$ when $k$ is a regular $\bF_p$-algebra in terms of the big de\,Rham-Witt forms of $k$. Here we use a similar approach but exploit new equivariant homotopy computations for the topological Hochschild homology of the integers to obtain the results in Theorem \ref{thm:main}. 

For a ring $A$, the cyclotomic trace map \cite{BHM}, $trc: K(A) \rightarrow \TC(A),$ relates the algebraic $K$-theory of $A$ to the topological cyclic homology of $A$. This map is often close to an equivalence \cite{HeMa97, Mc97, GeHe06}, so in good cases the computation of algebraic $K$-theory, $K_q(A)$ can be reduced to the computation of topological cyclic homology, $\TC_q(A)$. 

Topological cyclic homology is defined by looking at fixed points of topological Hochschild homology, $T(A)$. The circle $S^1$ acts on $T(A)$, and we define $\TR^n(A;p) = T(A)^{C_{p^{n-1}}}$ to be the fixed point spectrum under the action of the cyclic group of order $p^{n-1}$ considered as a subgroup of $S^1$. To compute topological cyclic homology it is essential to compute the homotopy groups of these spectra:
\[
\TR^n_q(A;p) = \pi_q(T(A)^{C_{p^{n-1}}}) =  [S^q \wedge S^1/C_{p^{n-1}+}, T(A)]_{S^1}.
\]
In the case where the ring in question is a pointed monoid algebra, $A(\Pi)$, we take advantage of the following equivalence of $S^1$-spectra
\[
T(A(\Pi)) \simeq T(A) \wedge B^{cy}(\Pi).
\] 
Here $B^{cy}(\Pi)$ denotes the cyclic bar construction on the pointed monoid $\Pi$. Then the TR-groups relevant to the calculation of $K_*(A(\Pi))$ are those of the form
\[
\TR^n_q(A(\Pi);p) = \pi_q(T(A(\Pi))^{C_{p^{n-1}}})= [S^q \wedge S^1/C_{p^{n-1}+}, T(A)\wedge B^{cy}(\Pi)]_{S^1}.
\]
If one can understand how $B^{cy}(\Pi)$ is built out of $S^1$-representation spheres, this gives a formula for these TR-groups in terms of groups of the form
\[
\TR^n_{q-\lambda}(A;p) = \pi_{q-\lambda}(T(A)^{C_{p^{n-1}}}) = [S^q \wedge S^1/C_{p^{n-1}+}, T(A)\wedge S^{\lambda}]_{S^1}.
\]
In other words, the computation of the ordinary ($\mathbb{Z}$-graded) TR-groups of $A(\Pi)$ can be reduced to the computation of the $RO(S^1)$-graded TR-groups of $A$. If these $RO(S^1)$-graded TR-groups can be computed, then in good situations information about the $K$-theory groups, $K_*(A(\Pi))$ can be recovered using the cyclotomic trace. 

To prove Theorem \ref{thm:main} we apply the approach outlined above as follows. We first use the cyclotomic trace map to relate the $K$-groups in question to certain birelative topological cyclic homology groups. Work of Hesselholt \cite[Theorem B]{He07} gives a formula for these topological cyclic homology groups in terms of $RO(S^1)$-graded TR-groups of $\bZ$, $\TR^n_{q-\lambda}(\mathbb{Z};p)$. We recall this formula in Section \ref{s:TR}. Thus the $K$-groups we are studying can be written in terms of these $RO(S^1)$-graded TR-groups. 

The proof of the main theorem is then reduced to computing TR-groups of the form $\TR^n_{q-\lambda}(\bZ;p)$. In Section \ref{s:proof} we make the computations necessary to finish the proof. 

Although we are unable to determine the group structure of the odd $K$-groups in general, in Section \ref{s:low} we compute the odd $K$-groups in some low degrees.

\section{From $K$-theory to TR-theory} \label{s:TR}

We begin by reducing the proof of the main theorem to the computation of certain $RO(S^1)$-graded TR-groups. To begin we would like to have a comparison theorem relating algebraic $K$-theory to topological cyclic homology via the cyclotomic trace, as described in the Introduction. Often such comparison theorems relate relative $K$-theory and relative topological cyclic homology. In this case, we need to consider birelative groups in order to get the desired comparison theorem. We first recall how the relevant birelative groups are defined.

For $A= k[x,y]/(xy),$ $B=k[x] \times k[y]$ the normalization of $A$, and $I = (x,y)$ the augmentation ideal, one could consider the diagram:
\begin{equation}\label{square} \xymatrix{ K(A) \ar[r] \ar[d] & K(A/I) \ar[d] \\ K(B) \ar[r] & K(B/I)} \end{equation}
The birelative algebraic $K$-theory group $K(A,B,I)$ is defined to be the iterated mapping fiber of this diagram. The analogous construction for topological cyclic homology would yield birelative TC.   There is a comparison theorem due to Geisser and Hesselholt \cite[Theorem 1]{GeHe06}, which says that the cyclotomic trace map induces an isomorphism
\[ K_q(A,B,I;\bZ/p^v) \to \TC_q(A,B,I;p,\bZ/p^v) \]
for each prime $p$ and any $v \geq 1$. For our computations we are interested in the relative $K$-group $K_q(A,I)$. Note that for a regular ring $k$ the bottom horizontal map in diagram \ref{square}, $K(k[x] \times k[y]) \rightarrow K(k \times k)$, is a weak equivalence by the Fundamental Theorem of $K$-theory. It follows that the relative group $K_q(A,I)$ is isomorphic to the birelative group $K_q(A,B,I)$. So, the above comparison theorem allows us to compare the group $K_q(A,I)$ to a topological cyclic homology group. We now focus our attention on computing $\TC_q(A,B,I;p,\bZ/p^v)$, when $k=\mathbb{Z}$.

We note that $\mathbb{Z}[x,y]/(xy)$ is a pointed monoid algebra.  We consider the pointed monoid $\Pi = \{0, 1, x, x^2, \dots, y, y^2, \ldots\}$.  Then $\mathbb{Z}[x,y]/(xy) = \mathbb{Z}(\Pi)$. As outlined in the Introduction, in order to study the algebraic $K$-theory of $\mathbb{Z}(\Pi)$ using the $RO(S^1)$-graded TR-groups of $\mathbb{Z}$, one must first understand the $S^1$-equivariant homotopy type of $B^{cy}(\Pi)$. Hesselholt studied this equivariant homotopy type in \cite[Section 1]{He07}, which lead to a splitting of birelative topological Hochschild homology
\[ T(k[x,y]/(xy), k[x]\times k[y],(x,y)) \simeq \bigvee_{i \geq 1} T(k) \sma S^{\lambda_i} \sma \Sigma^{-1} (S^1/C_i)_+.\]
Here $S^{\lambda_i}$ is the $1$-point compactification of the real $S^1$-representation
\[ \lambda_i = \bC(1) \oplus \ldots \oplus \bC(i), \]
where $\bC(i)$ denotes the one-dimensional complex $S^1$-representation defined by $\bC(i) = \bC$ with $S^1$ acting from the left by $z\cdot w = z^iw$. 

Recall that there is an action on birelative topological Hochschild homology by $S^1$. We write $\TR^n(A, B, I;p) = T(A,B,I)^{C_{p^{n-1}}}$  for the fixed point spectrum under the action of $C_{p^{n-1}} \subset S^1$, the cyclic group of order $p^{n-1}$. These TR-spectra are connected by maps $R, F$, and $V$. The birelative topological cyclic homology spectrum $\TC(A, B, I;p)$ is formed by taking a homotopy limit over $R, F: \TR^n(A, B, I;p) \to \TR^{n-1}(A, B, I;p)$. Hesselholt's splitting yields the following formula:
\[ \TC_q(k[x,y]/(xy),k[x]\times k[y],(x,y);p) \cong \prod_{p \nmid d} \lim_R \TR^n_{q-\lambda_{p^{n-1} d}}(k;p).\]
Hesseholt was studying the case when the ring $k$ is a regular $\bF_p$-algebra, but this formula is valid even without that hypothesis. 

The $K$- theory groups of $\mathbb{Z}[x,y]/(xy)$ are finitely generated. Therefore to prove the main theorem it suffices to show that $\TC_{2i}(\mathbb{Z}[x,y]/(xy),\mathbb{Z}[x]\times \mathbb{Z}[y],(x,y);p) \cong \bZ$ and $\TC_{2i+1}(\mathbb{Z}[x,y]/(xy),\mathbb{Z}[x]\times \mathbb{Z}[y],(x,y);p)_{(p)}$ has order the $p$-primary component of $(i!)^2$ for each prime $p$.

\section{Calculation of TR-groups} \label{s:proof}

In Section \ref{s:TR} above, we reduced the proof of the main theorem to the proof of the following proposition. 
\begin{proposition}\label{prop:both}
Let $n$ be a positive integer and let $i$ be a non-negative integer. Then for every prime number $p$: \\
(i) The $\mathbb{Z}_{(p)}$-module  $\prod \lim_R \TR^n_{2i-\lambda_{p^{n-1}d}}(\mathbb{Z};p)_{(p)}$ has rank 1. \\
(ii) The $\mathbb{Z}_{(p)}$-module  $\prod \lim_R \TR^n_{2i+1-\lambda_{p^{n-1}d}}(\mathbb{Z};p)_{(p)}$ has finite length $v_p((i!)^2)$. \\
Here both products are taken over the set of positive integers which are not divisible by $p$.
\end{proposition}

\begin{proof}
We fix a prime $p$ and let $I_p$ denote the set of positive integers not divisible by $p$.  Given a virtual representation $\alpha \in RO(S^1)$, there is a prime operation defined by $\alpha' = \rho_p^* \alpha^{C_p}$ where $\rho_p: S^1 \rightarrow S^1/C_p$ is the isomorphism given by the $p$th root. There is a fundamental long exact sequence for $RO(S^1)$-graded TR-groups \cite[Theorem 2.2]{HeMa97}:
$$
\xymatrix{
\cdots \ar[r] & \pi_{q-\lambda}T(\mathbb{Z})_{hC_{p^n}} \ar[r]  & \TR^n_{q-\lambda}(\mathbb{Z};p)  \ar[r]^R  & \TR^{n-1}_{q-\lambda'} (\mathbb{Z};p)  \ar[r]  &\cdots
}
$$ 
Here $T(\mathbb{Z})_{hC_{p^n}} = (ES^1_+ \wedge
T(\mathbb{Z}))^{C_{p^{n-1}}}$ is the homotopy orbit spectrum. There is a first quadrant spectral sequence 
$$
{E}^2_{s, t}= {H}_s(C_{p^{n-1}}, \pi_t(T(\mathbb{Z}) \wedge S^{\lambda})) \Rightarrow \pi_{s+t-\lambda}T(\mathbb{Z})_{hC_{p^n}}.
$$
It follows that $\pi_{q-\lambda}T(\mathbb{Z})_{hC_{p^n}} = 0$ when $q < \dim_{\mathbb{R}}(\lambda)$. Therefore the map
\[
R: \TR^n_{q -\lambda}(\mathbb{Z};p) \rightarrow \TR^n_{q - \lambda'}(\mathbb{Z};p)
\]
is an isomorphism when $q< \dim_{\mathbb{R}}(\lambda)$. Hence, in computing the limit
$$
\lim_R \TR^n_{q-\lambda_{p^{n-1}d}}(\mathbb{Z};p)
$$ we find that the maps become isomorphisms when $n$ is large enough that $q<\dim(\lambda_{p^{n-1}d}) = 2p^{n-1}d.$ If $s$ is the smallest integer such that $q< 2p^sd$, then
$$
\lim_R \TR^n_{q-\lambda_{p^{n-1}d}}(\mathbb{Z};p) \cong \TR^{s+1}_{q-\lambda_{p^{s}d}}(\mathbb{Z};p) 
$$
The restriction map $R:  \TR^{s+1}_{q-\lambda_{p^{s}d}}(\mathbb{Z};p)\rightarrow \TR^s_{q-\lambda_{p^{s-1}d}}(\mathbb{Z};p)$ is an isomorphism. Hence, the canonical projection
\[
\lim_R \TR^n_{q-\lambda_{p^{n-1}d}}(\mathbb{Z};p) \rightarrow \TR^s_{q-\lambda_{p^{s-1}d}}(\mathbb{Z};p)
\]
is an isomorphism for the unique $s$ with $2p^{s-1}d \leq q <2 p^s d$. So, we compute the group $\TR^s_{q-\lambda_{p^{s-1}d}}(\mathbb{Z};p)$ for this $s$.

We first we consider the TR-groups graded by even-dimensional representations. By \cite[Theorem B]{AnGeHe}, the group $\TR^s_{2i-\lambda_{p^{s-1} d}}(\bZ;p)$ is torsion free of rank equal to the number of integers $0 \leq m <s$ such that $i=\dim_\bC(\lambda_{p^{s-1} d}^{C_{p^m}})$. Since $\lambda_{p^{s-1}d}^{C_{p^m}} = \lambda_{p^{s-1-m}d}$, the group in question is torsion free of rank equal to the number of integers $0\leq m < s$ such that $i=p^{s-1-m}d$. Hence, the abelian group $\lim_R \TR^n_{2i-\lambda_{p^{r-1} d}}(\bZ;p)$, for $p$ not dividing $d$, is free of rank $1$ if $i=p^n d$ for some $n \geq 0$ and zero otherwise. It follows that the abelian group
\[ \prod_{d\in I_p} \lim_R \TR^n_{2i-\lambda_{p^{n-1} d}}(\bZ;p) \]
is free of rank $1$. This proves part (i) of Proposition \ref{prop:both}. 

We now prove part (ii) of the proposition. As above, $\lim_R \TR^n_{2i+1-\lambda_{p^{n-1}d}}(\mathbb{Z};p)$ is isomorphic to $ \TR^s_{2i+1-\lambda_{p^{s-1}d}}(\mathbb{Z};p)$ for the unique $s$ with $2p^{s-1}d \leq q <2 p^s d$. By \cite[Theorem B]{AnGeHe}, for any $n>0$,
\begin{multline*}
 |\TR^n_{2i+1-\lambda_{p^{n-1}d}}(\mathbb{Z};p)| = |\TR^{n-1}_{2i+1-\lambda_{p^{n-1}d}^{C_p}}(\mathbb{Z};p)| p^{n-1}(i+1 - \dim_{\mathbb{C}}(\lambda_{p^{n-1}d})). \\
   = |\TR^{n-1}_{2i+1-\lambda_{p^{n-2}d}}(\mathbb{Z};p)| p^{n-1}(i+1 - p^{n-1}d).
 \end{multline*}
Using this result and induction we find that
\[ | \TR^s_{2i+1-\lambda_{p^{s-1} d}}(\bZ;p) | = \prod_{0 \leq k \leq s-1} p^{k} \cdot \prod_{0 \leq k \leq s-1} (i+1-p^{k} d).
\]
Hence
\[
\textup{length}_{\mathbb{Z}_{(p)}}  \lim_R \TR^n_{2i+1-\lambda_{p^{n-1}d}}(\mathbb{Z};p)_{(p)} =  \displaystyle\sum_{0\leq k \leq s-1} (k + v_p(i+1 - p^{k}d)).
\]
Taking the product over $d$ not divisible by $p$, and recalling that $p^{s-1}d \leq i < p^sd$,  we find that the length of the $\mathbb{Z}_{(p)}$-module  $\prod \lim_R \TR^n_{2i+1-\lambda_{p^{n-1}d}}(\mathbb{Z};p)_{(p)}$ is 
\begin{multline*}
\displaystyle\sum_{\stackrel{1 \leq d \leq i}{d\in I_p}}\displaystyle\sum_{0\leq k \leq s-1} (k + v_p(i+1 - p^{k}d)) = \displaystyle\sum_{1 \leq n \leq i} (v_p(n)+ v_p(i+1 - n)) 
= 2\displaystyle\sum_{1 \leq n \leq i} v_p(n) 
\end{multline*}
Hence, the $\mathbb{Z}_{(p)}$-module 
\[ \prod_{d \in I_p} \lim_R \TR^n_{2i+1-\lambda_{p^{n-1} d}}(\bZ;p)_{(p)} \]
has finite length $v_p((i!)^2)$.

\end{proof}
This finishes the proof of Theorem \ref{thm:main}.

\section{Low-dimensional calculations} \label{s:low}
Now we turn to explicit computations of the $K$-groups in low degrees. We find the following:

\begin{theorem} \label{thm:lowdim}
Let $A=\bZ[x,y]/(xy)$ and let $I=(x,y)$. Then
\begin{eqnarray*}
K_1(A,I) & = & 0 \\
K_3(A,I) & = & 0 \\
K_5(A,I) & \cong & \bZ/4 \\
K_7(A,I) & \cong & \bZ/2 \oplus \bZ/2 \oplus \bZ/9 \\
K_{11}(A,I) & \cong & \bZ/2 \oplus \bZ/32 \oplus \bZ/9 \oplus \bZ/25 \\
\end{eqnarray*}
\end{theorem}

\begin{proof}
It follows immediately from Theorem \ref{thm:main} that $K_1(A,I)=K_3(A,I)=0$. For the rest of the proof we will freely use \cite[Theorem 1.4]{AnGe}, which tells us the number of summands of each $\TR$-group, and \cite[Theorem B]{AnGeHe}, which tells us the order of each $\TR$-group.

%
%

We find that $K_5(A,I) \cong \TR^2_{5-\lambda_2}(\bZ;2)$, which has order $4$ and one summand. Hence $K_5(A,I)$ is as claimed. Similarly,
\[ K_7(A,I) \cong \TR^2_{7-\lambda_2}(\bZ;2) \oplus \TR^2_{7-\lambda_3}(\bZ;3).\]
The $2$-primary part has order $4$ and two summands, and the $3$-primary part has order $9$ and one summand. Hence $K_7(A,I)$ is as claimed.

For $K_{11}(A,I)$, we find that
\[ K_{11}(A,I) \cong \TR^3_{11-\lambda_4}(\bZ;2) \oplus \TR^2_{11-\lambda_3}(\bZ;3) \oplus \TR^2_{11-\lambda_5}(\bZ;5).\]
We find that $\TR^2_{11-\lambda_3}(\bZ;3) \cong \bZ/9$ and $\TR^2_{11-\lambda_5}(\bZ;5) \cong \bZ/25$, while $\TR^3_{11-\lambda_4}(\bZ;2)$ is identified in Proposition \ref{p:TR11} below.

\end{proof}


To compute $\TR^3_{11-\lambda_4}(\bZ;2)$, we will use the Tate spectral sequence. Recall that we have the following fundamental diagram of horizontal long exact sequences \cite[Equation 49]{HeMa97}:

\begin{equation*}
\xymatrix{ \ldots \ar[r] & \pi_{q-\lambda} T_{hC_{p^n}} \ar[r]^-N \ar[d]^= & \TR^{n+1}_{q-\lambda} \ar[r]^-R \ar[d]^{\Gamma_n} & \TR^n_{q-\lambda'} \ar[d]^{\hat{\Gamma}_n} \ar[r] & \ldots \\
\ldots \ar[r] & \pi_{q-\lambda} T_{hC_{p^n}} \ar[r]^-{N^h} & \pi_{q-\lambda} T^{hC_{p^n}} \ar[r]^-{R^h} & \pi_{q-\lambda} T^{t C_{p^n}} \ar[r] & \ldots }
\end{equation*}

Here $T=T(\bZ)$ is the topological Hochschild spectrum, $T^{tC_{p^n}}$ is the Tate spectrum, $T^{hC_{p^n}}$ is the homotopy fixed point spectrum, $T_{hC_{p^n}}$ is the homotopy orbit spectrum, and $\TR^n=\TR^n(\bZ;p)$. As earlier, \cite{AnGe}, $\lambda'=\rho_p^*(\lambda^{C_p})$ where $\rho_p : S^1 \to S^1/C_p$ is the $p$'th root map. By Tsalidis' Theorem \cite[Theorem 2.4]{Ts98} extended to the $RO(S^1)$-graded context \cite[Addendum 9.1]{HeMa97}, the maps $\Gamma_n$ and $\hat{\Gamma}_n$ are isomorphisms for $q \geq 2\dim_\bC(\lambda')$.


The Tate spectral sequence, which computes $\pi_{*-\lambda} T^{tC_{p^n}}$, has $E_2$ term
\[ \hat{E}_2^{s,t} = \widehat{H}^s(C_{p^n}; \pi_{t-\lambda} T(\bZ)).\]
Here $\pi_{t-\lambda} T(\bZ) \cong \pi_{t-2\dim_\bC(\lambda)} T(\bZ)$. By restricting to the second quadrant, we get a corresponding spectral sequence which computes $\pi_{*-\lambda} T^{hC_{p^n}}$. These spectral sequences were studied in detail with mod $p$ coefficients in \cite{AnGe}. While understanding the Tate spectral sequence with integral coefficients remains an extremely difficult problem, an essential ingredient in the proof of \cite[Theorem B]{AnGeHe} is that all non-zero differentials go from even to odd total degree. 
A partial understanding of the $C_4$-Tate spectral sequence with mod $4$ coefficients will be enough to compute $\TR^3_{11-\lambda_4}(\bZ;2)$.




\begin{proposition} \label{p:TR11}
We have
\[ \TR^3_{11-\lambda_4}(\bZ;2) \cong \bZ/2 \oplus \bZ/32.\]
\end{proposition}

\begin{proof}
We know that $\TR^3_{11-\lambda_4}(\bZ;2)$ has order $2^6$ and consists of two summands. Hence it is enough to show that $\TR^3_{11-\lambda_4}(\bZ;2,\bZ/4)$ has order $8$. We do this by studying the Tate spectral sequence converging to $\pi_{*-\lambda_4}(T(\bZ)^{tC_4};\bZ/4)$ and restricting to the second quadrant.

We start by computing $\TR^2_{11-\lambda_2}(\bZ;2) \cong \bZ/8$, which means that
  \[ \pi_{11-\lambda_4}(T(\bZ)^{tC_4};\bZ/4) \cong \bZ/4.\]
When restricting to the second quadrant (which in this case means filtration less than or equal to $8$) we pick up one extra class for each differential entering the second quadrant. The only possible classes that can support such differentials with target in total degree $11$ are $t^{-6}$, $2t^{-6}$ and $t^{-5} u_2 \lambda_1$.

There is a differential $d_4(t^{-7} u_2)=t^{-5} u_2 \lambda_1$, so $t^{-5} u_2 \lambda_1$ does not give us an extra class. (In the corresponding Tate spectral sequence with integral coefficients the class $t^{-7} u_2$ does not exist, and $t^{-5} u_2 \lambda_1$ supports a differential giving an extra integral class.) We also know that there is a differential $d_{12}(t^{-6})=\lambda_1 \mu_1^2$. Hence it suffices to show that $2t^{-6}$ is a permanent cycle. (Again, $2t^{-6}$ supports a longer differential in the corresponding Tate spectral sequence with integral coefficients.)

The $C_2$-Tate spectral sequence with mod $4$ coefficients has been worked out by Rognes, and is (mostly) described in \cite[Fig.\ 4.3]{Ro99b}. In particular $t^{-4}$ is a permanent cycle. The $RO(S^1)$-graded $C_2$-Tate spectral sequence is a shifted copy of the integral one, and it follows that $t^{-6}$ is a permanent cycle in the spectral sequence converging to $\pi_{*-\lambda_4}(T(\bZ)^{tC_2};\bZ/4)$.

Next we use that the Verschiebung (or transfer) map $V : T(\bZ)^{tC_2} \to T(\bZ)^{tC_4}$ gives a map of spectral sequences, and that $V(t^{-6})=2t^{-6}$. It follows that $2t^{-6}$ is a permanent cycle in the spectral sequence converging to $\pi_{*-\lambda_4}(T(\bZ)^{tC_4};\bZ/4)$, which was what we needed to show.

 \end{proof}

We also conjecture the computation of two other $K$-theory groups.

\begin{conjecture} \label{conj:lowdim}
Let $A=\bZ[x,y]/(xy)$ and let $I=(x,y)$. Then
\begin{eqnarray*}
K_9(A,I) & \cong & \bZ/2 \oplus \bZ/2 \oplus \bZ/16 \oplus \bZ/3 \oplus \bZ/3 \\
K_{13}(A,I) & \cong & \bZ/2 \oplus \bZ/2 \oplus \bZ/8 \oplus \bZ/8 \oplus \bZ/3 \oplus \bZ/3 \oplus \bZ/9 \oplus \bZ/5 \oplus \bZ/5 \\
\end{eqnarray*}
\end{conjecture}

\begin{proof}[Sketch proof]
This conjecture is based on a conjectural understanding of the $C_4$-Tate spectral sequence with integral coefficients. We compute that
\[ K_9(A,I) \cong \TR^3_{9-\lambda_4}(\bZ;2) \oplus \TR^1_{9-\lambda_3}(\bZ;2) \oplus \TR^2_{9-\lambda_3}(\bZ;3) \oplus \TR^1_{9-\lambda_2}(\bZ;3).\]
We find that $\TR^1_{9-\lambda_3}(\bZ;2) \cong \bZ/2$, $\TR^2_{9-\lambda_3}(\bZ;3) \cong \bZ/3$ and $\TR^1_{9-\lambda_2}(\bZ;3) \cong \bZ/3$. The final group, $\TR^3_{9-\lambda_4}(\bZ;2)$, is conjecturally computed in Conjecture \ref{c:TR9} below.

Similarly we compute
\begin{eqnarray*}
K_{13}(A,I) \cong \TR^3_{13-\lambda_4}(\bZ;2) \oplus \TR^2_{13-\lambda_6}(\bZ;2) \oplus \TR^1_{13-\lambda_5}(\bZ;2) \oplus \TR^2_{13-\lambda_3}(\bZ;3) \\
\oplus \TR^2_{13-\lambda_6}(\bZ;3) \oplus \TR^1_{13-\lambda_4}(\bZ;3) \oplus \TR^2_{13-\lambda_5}(\bZ;5) \oplus \TR^1_{13-\lambda_2}(\bZ;5).
\end{eqnarray*}
We conjecture the group structure of $\TR^3_{13-\lambda_4}(\mathbb{Z};2)$ in Conjecture \ref{c:TR13} below, and the other groups are readily computed.
\end{proof}


\begin{conjecture} \label{c:TR9}
 We have
 \[ \TR^3_{9-\lambda_4}(\bZ;2) \cong \bZ/2 \oplus \bZ/16.\]
 \end{conjecture}

\begin{proof}[Sketch proof]
We know that $\TR^3_{9-\lambda_4}(\bZ;2)$ has order $2^5$ and consists of two summands. To compute this term we study the Tate spectral sequence converging to $\pi_{*-\lambda_4} T(\bZ)^{tC_4}$ and restrict to the second quadrant. (Which, again, means filtration less than or equal to $8$.) We start by computing $\TR^2_{9-\lambda_2}(\bZ;2) \cong \bZ/8$. Now we conjecture that the $C_4$-Tate spectral sequence in total degree $9$ and $10$ behaves as in Figure \ref{TateSS1}. The $d_4$ and $d_5$ differentials follow by comparing with the $C_2$-Tate spectral sequence, but we have not been able to rule out that the longer differentials could behave in a more complicated way. In the Tate spectral sequence converging to $\pi_{9-\lambda_4} T(\bZ)^{tC_4} \cong \bZ/8$, the following classes should survive:
 \[ \{t^{-1} \lambda_1 \mu_1, t^3 \lambda_1 \mu_1^3, t^7 \lambda_1 \mu_1^5\}. \]
These are then connected by hidden multiplication by $2$ extensions.

\begin{figure}\label{figure1}
 \begin{center}
\includegraphics[scale=0.8]{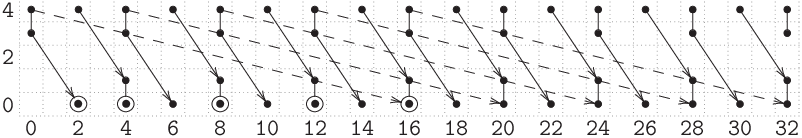} \caption{Degree $9$ and $10$ of the Tate spectral sequence converging to $\pi_{*-\lambda_4} T(\bZ)^{tC_4}$} \label{TateSS1}
 \end{center}
\end{figure}

When restricting to the second quadrant we pick up two more classes, coming from the differentials originating from $t^{-5}$ and $2t^{-5}$. So we would have the following classes, circled in Figure \ref{TateSS1}:
 \[ \{ t^{-3} \lambda_1, t^{-1} \lambda_1 \mu_1, t^3 \lambda_1 \mu_1^3, t^7 \lambda_1 \mu_1^5, t^{11} \lambda_1 \mu_1^7 \}.\]

Now consider the class $t^4 u_2 \lambda_1 \mu_1^4$. This class kills $2t^7 \lambda_1 \mu_1^5$, while in the corresponding spectral sequence with mod $2$ coefficients it kills $t^{11} \lambda_1 \mu_1^7$. This implies that we have a hidden multiplication by $2$ extension connecting the classes $t^7 \lambda_1 \mu_1^5$ and $t^{11} \lambda_1 \mu_1^7$. Hence the class $t^{-1} \lambda_1 \mu_1$ has order $16$ and $\TR^3_{9-\lambda_4}(\bZ;2) \cong \bZ/2 \oplus \bZ/{16}$.

\end{proof}


\begin{conjecture} \label{c:TR13}
 We have
 \[ \TR^3_{13-\lambda_4}(\bZ;2) \cong \bZ/2 \oplus \bZ/8.\]
\end{conjecture}

The argument in support of this conjecture is similar to the one for Conjecture \ref{c:TR9}.

\bibliographystyle{plain}
\bibliography{b}

\end{document}